\documentclass[11pt,twoside]{amsart}

\usepackage{amsmath}
\usepackage{amsthm}
\usepackage{amsfonts, amssymb}
\usepackage{mathrsfs}
\usepackage[all]{xy}
\usepackage{url}

\usepackage{latexsym}
\usepackage[dvips]{graphicx}

\theoremstyle{plain}
\newtheorem{lema}{Lemma}[section]
\newtheorem{prop}[lema]{Proposition}
\newtheorem{teo}[lema]{Theorem}

\newtheorem{coro}[lema]{Corollary}
\newtheorem*{ida}{Proposition \ref{idaa}}

\newtheorem*{col}{Theorem \ref{coll}}
\theoremstyle{remark}

\newtheorem{obs}[lema]{Remark}

\theoremstyle{definition}
\newtheorem{defi}[lema]{Definition}
\newtheorem{ej}[lema]{Example}

\def\gammac{\searrow \! \! \! \! \! ^{\gamma} \: \: }
\def\ce{\mbox{$\searrow \! \! \! \! \! ^e \: \: $}}

\def\se{\mbox{$\diagup \! \! \! \! \: \! \: \searrow$}}
\def\k{\mathcal{K}}
\def\x{\mathcal{X}}

\def\etft0{\mathnormal{etfT_0}}

\def\S{\mathbb{S}}

\begin{document}

\title[One point reductions, h-regular CW-complexes and collapsibility]{One-point reductions of finite spaces, \\ h-regular CW-complexes and collapsibility}

\author[J.A. Barmak]{Jonathan Ariel Barmak}
\author[E.G. Minian]{Elias Gabriel Minian}

\address{Departamento  de Matem\'atica\\
 FCEyN, Universidad de Buenos Aires\\ Buenos
Aires, Argentina}

\email{jbarmak@dm.uba.ar}
\email{gminian@dm.uba.ar}

\begin{abstract}
We investigate one-point reduction methods of finite topological spaces. These methods allow one to study homotopy theory of cell complexes by means of elementary moves of their finite models. We also introduce the notion of h-regular CW-complex, generalizing the concept of regular CW-complex, and prove that the h-regular CW-complexes, which are a sort of combinatorial-up-to-homotopy objects, are modeled (up to homotopy) by their associated finite spaces. This is accomplished by generalizing a classical result of McCord on simplicial complexes. 
\end{abstract}



\maketitle

\section{Introduction}

%There exists a natural correspondence between topologies and preorders defined on a finite set $X$, which was originally studied by Alexandroff \cite{Ale}. Under this correspondence, $T_0$-topologies can be regarded as order relations.
 Two independent and foundational papers on finite spaces of 1966, by M.C. McCord and R.E. Stong \cite{Mcc,Sto}, investigate the homotopy theory of finite spaces and their relationship with polyhedra.  McCord \cite{Mcc}  associates to a finite simplicial complex $K$, the finite $T_0$-space $\x(K)$ which corresponds to the poset of simplices of $K$ and proves that there is a weak homotopy equivalence $K\to \x(K)$. Conversely, one can associate to a given finite $T_0$-space $X$ the simplicial complex $\k(X)$ of its non-empty chains and a weak homotopy equivalence $\k(X)\to X$. In contrast to McCord's approach, Stong introduces a combinatorial method to describe the homotopy types of finite spaces. He defines the notions of \it linear \rm and \it colinear \rm points, which we call \it down \rm and \it up beat points \rm following Peter May's terminology,  and proves that two finite spaces have the same homotopy type if and only if one of them can be obtained from the other by adding or removing beat points.
Recently a series of notes by Peter May \cite{May,May2} caught our attention to finite spaces. In his notes, May discusses various basic problems from the perspective of finite spaces.
It is evident, from McCord's and Stong's papers and from May's notes, that finite topological spaces can be used to develop new techniques, based on their combinatorial and topological nature, to investigate homotopy theory of polyhedra. A nice example of this is Stong's paper of 1984 \cite{Sto2} where he restates Quillen's conjecture on the poset of non-trivial $p$-subgroups of a group \cite{Qui} in terms of finite spaces. Also in this direction, we showed in \cite{Bar2} how to use finite spaces to study simple homotopy types.

This article deals with one-point reductions. We investigate the cases in which removing a particular point of the space does not affect its homotopy, weak homotopy or simple homotopy 
type.
Our starting point is Theorem 3.10 of \cite{Bar2},  which relates simplicial collapses with collapses of finite spaces. More explicitly, we have proved in \cite{Bar2} that a collapse $X\searrow  Y$ between 
finite spaces induces a collapse $\k(X) \searrow \k(Y)$ between their associated
simplicial complexes and  a simplicial collapse $K \searrow L$ induces a collapse between the associated finite spaces. One advantage of working with finite spaces is that the elementary collapses in this context are very simple to handle and describe: 
they consist of removing a single point of the space,  which is called  a \it weak point. \rm The beat points introduced by Stong \cite{Sto} and the weak points defined in \cite{Bar2} constitute particular cases of one-point reductions. The main idea that is behind the one-point reductions is the idea of an \it elementary move, \rm which appears frequently in mathematics. Tietze transformations and Whitehead's theory on simple homotopy types are two leading exponents of this concept. The results that we obtain allow one to study homotopy theory of cell complexes by means of elementary moves of their \it finite models. \rm  A finite model of a CW-complex $K$ is a finite space which is weak homotopy equivalent to $K$.

In this paper we introduce the notions of $\gamma$-point and $\gamma$-collapse which provide a more general method of reduction. More precisely, we prove below the following

\begin{ida}
If $x\in X$ is a $\gamma$-point, the inclusion $i:X\smallsetminus \{x\} \hookrightarrow X$ is a weak homotopy equivalence.
\end{ida}

This also improves an old result which appears for example in  \cite[Proposition 5.8]{Wal}. Moreover, we prove that the converse of Proposition \ref{idaa} also holds provided $x$ is neither maximal nor minimal (see Theorem \ref{vueltaa}). Therefore, the elementary $\gamma$-collapses describe almost all possible one-point reductions.

%Although a $\gamma$-expansion does not induce in general a simplicial expansion, it does induce a simple homotopy equivalence.
We also investigate collapsibility and $\gamma$-collapsibility of the joins $X\oplus Y$ of finite spaces in terms of the collapsibility of $X$ and $Y$. This sheds some light on the analogous problem for simplicial joins.

In the last section of the paper we introduce the concept of \textit{h-regular CW-complex}.
Recall that a CW-complex is called regular if the characteristic maps of its cells are homeomorphisms. It is known that if $K$ is a regular CW-complex, its face poset $\x (K)$ is a finite model 
of $K$. For general CW-complexes the associated finite space $\x (K)$ does not give relevant information about the topology of $K$. Regular CW-complexes can be thought of as combinatorial objects (in fact, they are in the middle way between simplicial complexes and general CW-complexes). This suggests that one can extend the class of combinatorial CW-complexes to a wider class of combinatorial-up-to-homotopy CW-complexes  that can be modeled, up to homotopy, by their associated finite spaces. This leads to the notion of \textit{h-regular CW-complex}.
 A CW-complex is h-regular if all its closed cells are contractible subcomplexes. In particular, regular 
complexes are h-regular.

We prove that if $K$ is a finite h-regular complex, there is a weak homotopy equivalence $K\to \x (K)$, generalizing McCord's result  
for finite simplicial complexes (compare with \cite{Bjo}).
The paper ends with the following result which relates collapses of h-regular complexes with  $\gamma$-collapses.
\begin{col}
Let $L$ be a subcomplex of an h-regular complex $K$. If $K\searrow L$, then $\x(K) \gammac \x(L)$.
\end{col}

\section{Preliminaries}

In this section we recall the basic notions on finite spaces and their relationship with finite posets and simplicial complexes. For more details, we refer the reader to \cite{Bar2,May2, Mcc,Sto}.

Given a $T_0$-topology $\tau$ on a finite set $X$, we define for each $x\in X$ the (open) set $U_x$ as the intersection of all  open sets containing $x$. Recall that a topological space $X$ is called $T_0$ if for every pair of points in $X$ there exists some open set containing one and only one of them. The order associated to the $T_0$-topology $\tau$ is given by $x\le y$ if $x\in U_y$. 
This application establishes a one to one correspondence between $T_0$-topologies and order relations on the set $X$. Therefore we can regard finite $T_0$-spaces as finite posets and viceversa. 
It is not hard to see that a function is continuous if and only if it is order preserving. Moreover if $f,g:X\to Y$ are two maps such that $f(x)\le g(x)$ for every $x\in X$, they are homotopic.

The \textit{order complex} $\k (X)$ of a finite $T_0$-space $X$ is the simplicial complex whose simplices are the non-empty chains of $X$. It is also denoted $\Delta (X)$ by some authors. There exists a weak homotopy equivalence from the geometric realization $|\k (X)|$ to $X$, i.e. a map which induces isomorphisms in all homotopy groups \cite{Mcc}.

\begin{ej}
Consider the finite $T_0$-space $X$ represented by the Hasse diagram shown in Fig. 1. The order complex of $X$ is homeomorphic to the M\"obius Strip.

\begin{minipage}{4cm}
\begin{displaymath}
\xymatrix@C=14pt{ \bullet \ar@{-}[d] \ar@{-}[dr] & \bullet \ar@{-}[ld] \ar@{-}[rd] & \bullet \ar@{-}[ld] \ar@{-}[d] \\
								^a \bullet \ \ar@{-}[d] \ar@{-}[dr] & ^b \bullet \ \ar@{-}[ld] \ar@{-}[rd] & \ \bullet ^c \ar@{-}[ld] \ar@{-}[d] \\
		\bullet & \bullet & \bullet} 
\end{displaymath}
\center{\raisebox{-0.5ex}{Fig. 1: Hasse diagram}}
\center{of $X$}
\end{minipage}
\begin{minipage}{10cm}
\smallskip
\smallskip
\smallskip
\includegraphics[scale=0.4]{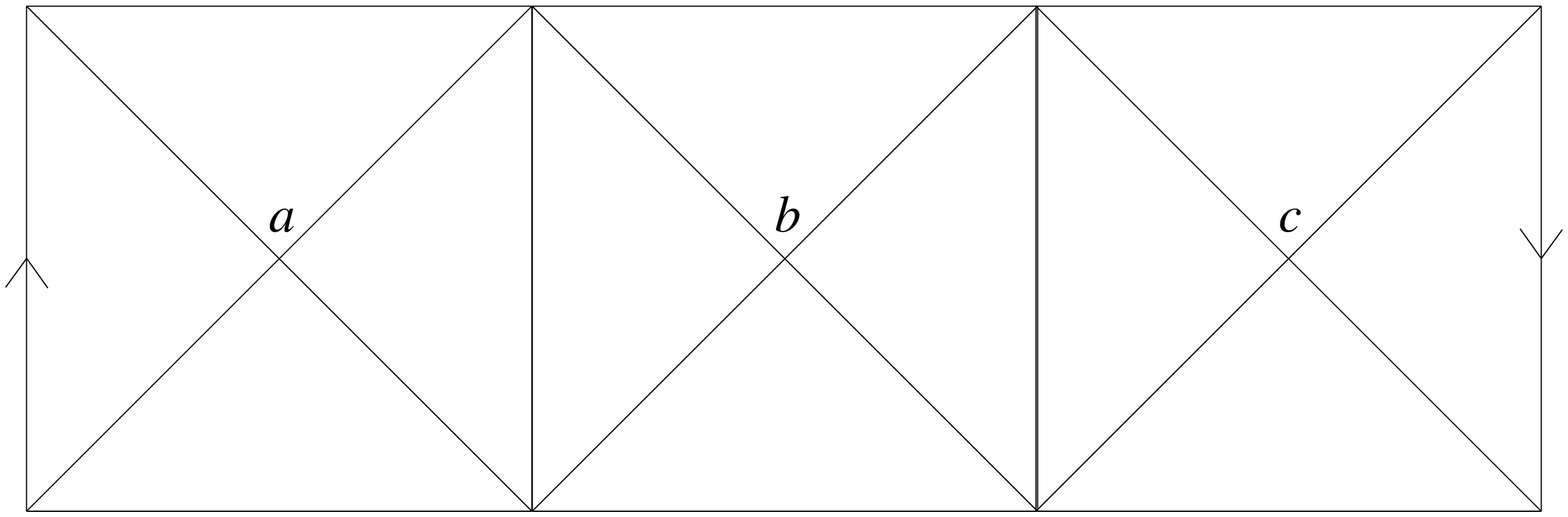}
\center{\raisebox{0.1ex}{Fig. 2: $\k (X)$ is a triangulation of the M\"obius Strip}}
\center{with twelve 2-simplices}
\end{minipage}
\end{ej}

The application $\k$ is also defined on maps, moreover if $f:X\to Y$ is a map between finite $T_0$-spaces, there is a commutative diagram

\begin{displaymath}
\xymatrix@C=20pt{ |\k (X)| \ar@{->}[d] \ar@{->}^{|\k(f)|}[r] & |\k (Y)| \ar@{->}[d] \\
									X \ar@{->}^f[r] & Y }
\end{displaymath}

Conversely one can associate a finite space (the \textit{face poset}) $\x (K)$ to each finite simplicial complex $K$ which is the poset of simplices of $K$ ordered by inclusion. 
Since $\k (\x(K))=K'$ is the barycentric subdivision of $K$, there exists a weak homotopy equivalence $|K|\to \x (K)$.

Let $X$ be a finite $T_0$-space. A point $x\in X$ is a \textit{down beat point} if the set $\hat{U}_x=U_x\smallsetminus \{x\}$ of points smaller than $x$ has a maximum, and it is an \textit{up beat point} if the set $\hat{F}_x=F_x\smallsetminus \{x\}$ of points which are greater than $x$ has a minimum. Here $F_x$ denotes the closure of $\{x\}$ in $X$. 
If $x$ is a beat point (down or up), $X\smallsetminus \{x\} \hookrightarrow X$ is a strong deformation retract. Moreover, $X$ is contractible(=dismantlable poset) if and only if one can remove beat points one at the time to obtain a space of one point \cite{Sto}.

Following \cite{Bar2}, we say that a point $x\in X$ is a \textit{weak point} if $\hat{U}_x$ or $\hat{F}_x$ is contractible. Note that this definition generalizes the definition of a beat point since any finite space with maximum or minimum is contractible.
In this case, the inclusion $X\smallsetminus \{x\} \hookrightarrow X$ need not be a homotopy equivalence, but it is a weak homotopy equivalence. Note that in the context of finite spaces, weak homotopy equivalences are not in general homotopy equivalences.

The notion of weak point gives rise to the following notion of collapse for finite spaces.

\begin{defi}
If $x\in X$ is a weak point, we say that $X$ \textit{collapses} to $X\smallsetminus \{x\}$ by an \textit{elementary collapse}. We denote this by $X\ce X\smallsetminus \{x\}$.
We say that $X$ \textit{collapses} to $Y$ (or $Y$ expands to $X$), and write $X\searrow Y$, if there is a sequence of elementary collapses which starts in $X$ and ends in $Y$. The space $X$ is \textit{collapsible} if it collapses to a point.
Finally, $X$ and $Y$ are \textit{simply equivalent}, denoted by $X\se Y$, if there exists a sequence of collapses and expansions that starts in $X$ and ends in $Y$. 
\end{defi}

In contrast with the classical situation, where a simple homotopy equivalence is a special kind of homotopy equivalence, homotopy equivalent finite spaces are simply equivalent.
The relationship between collapses of finite spaces and simplicial complexes is given by the following

\begin{teo}
\begin{enumerate} 
\item[ ]
\item[(a)] Let $X$ and $Y$ be finite $T_0$-spaces. Then, $X$ and $Y$ are simply equivalent if and only if $\k(X)$ and $\k (Y)$ have the same simple homotopy type. 
Moreover, if $X \searrow Y$ then $\k (X) \searrow \k(Y)$.
\item[(b)] Let $K$ and $L$ be finite simplicial complexes. Then, $K$ and $L$ are simple homotopy equivalent if and only if $\x(K)$ and $\x (L)$ are simply equivalent. Moreover, if $K \searrow L$ then $\x (K) \searrow \x (L)$.
\end{enumerate}
\end{teo}

The proof of this theorem can be found in \cite{Bar2}.
In this paper we give an alternative proof of the fact that a collapse of finite spaces induces a collapse between the associated complexes.

\section{$\gamma$-points and reduction methods}

In this section we delve deeper into the study of one-point reductions of finite spaces. As we pointed out above, we investigate the cases in which removing a particular point from a finite space does not affect its homotopy, weak homotopy or simple homotopy type.
%We introduce the notion of a $\gamma$-point which gives rise to the notion of $\gamma$-collapse and relate this with the concept of collapse of %finite spaces and of simplicial complexes.

Recall that the \textit{simplicial join} $K*L$ of two simplicial complexes $K$ and $L$  is the complex
$$K*L=K\cup L\cup \{\sigma\cup \tau| \ \sigma\in K, \tau\in L\}.$$
The cone $aK$ of a simplicial complex $K$ is the join of $K$ with a vertex $a\notin K$.
It is well known that for finite simplicial complexes $K$ and $L$, the geometric realization $|K*L|$ is homeomorphic to the topological join $|K|*|L|$.

There is an analogous construction for finite spaces.

\begin{defi}
The \textit{(non-Hausdorff) join} $X\oplus Y$ of two finite $T_0$-spaces $X$ and $Y$ is the disjoint union $X\sqcup Y$ keeping the giving ordering within $X$ and $Y$ and setting $x\leq y$ for every $x\in X$ and $y\in Y$. 
\end{defi}

Special cases of joins are the non-Hausdorff cone $\mathbb{C} (X)=X\oplus D^0$ and  the non-Hausdorff suspension $\S (X)=X\oplus S^0$   of any finite $T_0$-space $X$. Here $D^0$ denotes the singleton (0-cell) and $S^0$ the discrete space on two points (0-sphere).

\begin{obs}
$\k (X\oplus Y)=\k (X)* \k (Y)$.
\end{obs}

Given a point $x$ in a finite $T_0$-space $X$, the \textit{star} $C_x$ of $x$ consists of the points which are comparable with $x$, i.e. $C_x=U_x\cup F_x$. Note that $C_x$ is always contractible since $1_{C_x}\le f \ge g$ where $f:C_x\to C_x$ is the map which is the identity on $F_x$ and the constant map $x$ on $U_x$, and $g$ is the constant map $x$.
The \textit{link} of $x$ is the subspace  $\hat{C}_x=C_x\smallsetminus \{x\}$. In case we need to specify the ambient space $X$, we will write $\hat{C}_x^X$. Note that $\hat{C}_x=\hat{U}_x \oplus \hat{F}_x$.

\begin{prop} \label{joincontr}
Let $X$ and $Y$ be finite $T_0$-spaces. Then $X \oplus Y$ is contractible if and only if  $X$ or $Y$ is contractible.
\end{prop}
\begin{proof}
Assume $X$ is contractible. Then there exists a sequence of spaces $$X=X_n\supsetneq X_{n-1} \supsetneq \ldots \supsetneq X_1=\{x_1\}$$ with $X_i=\{x_1,x_2, \ldots ,x_i\}$ and such that $x_i$ is a beat point of $X_i$ for every $2\le i \le n$. Then $x_i$ is a beat point of $X_i \oplus Y$ for each $2 \le i \le n$ and therefore, $X\oplus Y$ deformation retracts to $\{x_1\} \oplus Y$ which is contractible. Analogously, if $Y$ is contractible, so is $X\oplus Y$.

Now suppose $X\oplus Y$ is contractible. Then there exists a sequence  $$X\oplus Y=X_n\oplus Y_n \supsetneq X_{n-1} \oplus Y_{n-1} \supsetneq \ldots \supsetneq X_1 \oplus Y_1=\{z_1\}$$ with $X_i\subseteq X$, $Y_i\subseteq Y$, $X_i\oplus Y_i=\{z_1,z_2 \ldots ,z_i\}$ such that $z_i$ is a beat point of $X_i\oplus Y_i$ for $i\ge 2$.

Let $i\ge 2$. If $z_i\in X_i$, $z_i$ is a beat point of $X_i$ unless it is a maximal point of $X_i$ and $Y_i$ has a minimum. In the same way, if $z_i \in Y_i$, $z_i$ is a beat point of $Y_i$ or $X_i$ has a maximum. Therefore, for each $2\le i \le n$, either $X_{i-1}\subseteq X_i$ and  $Y_{i-1}\subseteq Y_i$ are deformation retracts (in fact, one inclusion is an identity and the other inclusion is strict), or one of them,  $X_i$ or $Y_i$, is contractible. This proves that $X$  or $Y$ is contractible.
\end{proof}

\begin{coro} \label{linkcontr}
Let $X$ be a finite $T_0$-space. Then $x\in X$ is a weak point if and only if its link $\hat{C_x}$ is contractible.
\end{coro}

In \cite{Bar2} we proved that a collapse $X\searrow Y$ of finite spaces induces a simplicial collapse $\k(X) \searrow \k(Y)$. 
We exhibit here an alternative proof of this result, using Corollary \ref{linkcontr} and the following easy lemma whose proof we omit.

\begin{lema} \label{mig}
Let $aK$ be a simplicial cone of a finite complex $K$. Then $K$ is collapsible if and only if $aK\searrow K$.
\end{lema}

We study first a particular case (cf. \cite[Theorem 3.3]{Osa}).

\begin{teo}\label{osaki}
If $x$ is a beat point of a finite $T_0$-space $X$, then $\k (X) \searrow \k (X \smallsetminus \{x\})$. In particular, if $X$ is contractible, $\k (X)$ is collapsible.
\end{teo}  
\begin{proof}
Since $x$ is a beat point, there exists $x'\in X$ and subspaces $Y,Z\subseteq X$ such that $\hat{C}_x=Y\oplus \{x'\} \oplus Z$. Then the link $lk(x)$ of the vertex $x$ in $\k(X)$ is collapsible, since  $lk (x)= \k (\hat{C}_x) =x'\k(Y\oplus Z)$. By the previous lemma, the star  $st(x)=xlk(x)$ collapses to  $lk (x)=\k (X\smallsetminus \{x\})\cap st(x)$. Thus, $\k(X)=\k (X\smallsetminus \{x\})\cup st(x)\searrow \k (X\smallsetminus \{x\})$.
\end{proof}

\begin{teo}
If $X\searrow Y$, then $\k (X) \searrow \k (Y)$.
\end{teo}
\begin{proof}
We may assume that $Y=X\smallsetminus \{x\}$, where $x\in X$ is a weak point. By Corollary \ref{linkcontr}, $\hat{C}_x$ is contractible and then $\k (\hat{C}_x)$ is collapsible. Now the result follows as in the proof of Theorem \ref{osaki}. 
\end{proof}

Note that if $\hat{C}_x$ is collapsible (but not necessarily contractible), we also have that $\k (X)\searrow \k(X\smallsetminus \{x\})$.

It is known that if $K$ and $L$ are finite simplicial complexes and one of them is collapsible, then $K*L$ is also collapsible. As far as we know the converse of this result
 is an open problem (see \cite[(4.1)]{Wel}). In the setting of finite spaces, the analogous result and its converse hold. 

\begin{prop} \label{joincol}
Let $X$ and $Y$ be finite $T_0$-spaces. Then $X \oplus Y$ is collapsible if and only if $X$ or $Y$ is collapsible.
\end{prop}
\begin{proof}
We proceed as in Proposition \ref{joincontr}, replacing beat points by weak points and deformation retractions by collapses. Note that if $x_i$ is a weak point of $X_i$, then $x_i$ 
is also a weak point of $X_i\oplus Y$, since $\hat{C}_{x_i}^{X_i\oplus Y}=\hat{C}_{x_i}^{X_i}\oplus Y$ is contractible by Proposition \ref{joincontr}.

On the other hand, if $z_i$ is a weak point of $X_i \oplus Y_i$ and $z_i\in X_i$, then by Proposition \ref{joincontr}, $z_i$ is a weak point of $X_i$ or $Y_i$ is contractible.
\end{proof}

Corollary \ref{linkcontr} motivates the following definition.

\begin{defi}
A point $x$ of a finite $T_0$-space $X$ is a \textit{$\gamma$-point} if $\hat{C}_x$ is homotopically trivial (i.e. if all its homotopy groups are trivial). 
\end{defi}

Note that weak points are $\gamma$-points. It is not difficult to see that both notions coincide in spaces of height less than or equal to $3$. This is  because any space of height 2 is 
contractible if and only if it is homotopically trivial. However, this is false for spaces of height greater than $3$.

Let $x$ be a $\gamma$-point of a finite $T_0$-space $X$. We will show that the inclusion $X\smallsetminus \{x\}\hookrightarrow X$ is a weak homotopy equivalence. Note that since $\hat{U}_x$ and $\hat{F}_x$ need not be homotopically trivial, we cannot proceed as we did in \cite{Bar2}. However, in this case, one has the following pushout

\begin{displaymath}
\xymatrix@C=20pt{ |\k (\hat{C}_x)| \ar@{->}[r] \ar@{->}[d] & |\k(C_x)| \ar@{->}[d] \\
									|\k(X\smallsetminus \{x\})| \ar@{->}[r] & |\k(X)| } 
\end{displaymath}
Where $|\k (\hat{C}_x)| \to |\k(C_x)|$ is a homotopy equivalence and $|\k (\hat{C}_x)|\to |\k(X\smallsetminus \{x\})|$ satisfies the homotopy extension property. Therefore $|\k (X\smallsetminus \{x\})| \to |\k(X)|$ is a homotopy equivalence. This proves the following

\begin{prop} \label{idaa}
If $x\in X$ is a $\gamma$-point, the inclusion $i:X\smallsetminus \{x\} \hookrightarrow X$ is a weak homotopy equivalence.
\end{prop}

As we mentioned in the introduction, this result improves an old result which appears for example in Walker's Thesis \cite[Proposition 5.8]{Wal}, which asserts, in the language of finite spaces, that $X\smallsetminus \{x\} \hookrightarrow X$ is a weak homotopy equivalence provided $\hat{U}_x$ or $\hat{F}_x$ is homotopically trivial. By Proposition \ref{join} below, it is clear that a point $x$ is a $\gamma$-point if $\hat{U}_x$ or $\hat{F}_x$ is homotopically trivial, but the converse is false. 

We will show that the converse of Proposition \ref{idaa} is true in most cases. First, we need some results.

\begin{prop} \label{homo}
Let $x$ be a point of a finite $T_0$-space $X$. The inclusion $i:X\smallsetminus \{x\} \hookrightarrow X$ induces isomorphisms in all homology groups if and only if the subspace $\hat{C}_x$ is acyclic.
\end{prop}
\begin{proof}
Apply the Mayer-Vietoris sequence to the triple $(\k (X) ; \k(C_x), \k(X\smallsetminus \{x\}))$.
\end{proof}

\begin{obs} \label{libre}
If $X$ and $Y$ are non-empty finite $T_0$-spaces with $n$ and $m$ connected components respectively, the fundamental group $\pi_1(X\oplus Y)$ is a free product of $(n-1)(m-1)$ copies of $\mathbb{Z}$. In particular if $x\in X$ is neither maximal nor minimal, the fundamental group of $\hat{C}_x=\hat{U}_x \oplus \hat{F}_x$ is a free group.
\end{obs}

\begin{teo} \label{vueltaa}
Let $X$ be a finite $T_0$-space, and $x\in X$ a point which is neither maximal nor minimal and such that $X\smallsetminus \{x\} \hookrightarrow X$ is a weak homotopy equivalence. Then $x$ is a $\gamma$-point.
\end{teo}
\begin{proof}
If $X\smallsetminus \{x\} \hookrightarrow X$ is a weak homotopy equivalence, $\hat{C}_x$ is acyclic by Proposition \ref{homo}. Then $\pi_1(\hat{C}_x)$ is a perfect group and therefore trivial, by Remark \ref{libre}. Now the result follows from the Hurewicz Theorem. 
\end{proof}

The theorem fails if $x$ is maximal or minimal as the next example shows.

\begin{ej}
Let $X$ be an acyclic finite $T_0$-space with non-trivial fundamental group. Let $\S (X)=X\cup \{-1,1\}$ be its non-Hausdorff suspension. Then $\S (X)$ is also acyclic and $\pi_1(\S (X))=0$. Therefore it is homotopically trivial. Hence, $X\cup \{1\} \hookrightarrow \S (X)$ is a weak homotopy equivalence. However $-1$ is not a $\gamma$-point of $\S (X)$.
\end{ej}
 
Using the relativity principle of simple homotopy theory \cite[(5.3)]{Coh} one can prove that if $x$ is a $\gamma$-point, $|\k (X\smallsetminus \{x\})| \to |\k(X)|$ is a simple homotopy equivalence. In fact this holds whenever $X\smallsetminus \{x\} \hookrightarrow X$ is a weak homotopy equivalence. 
 
\begin{teo} \label{inducesimple}
Let $X$ be a finite $T_0$-space and let $x\in X$. If the inclusion $i:X\smallsetminus \{x\} \hookrightarrow X$ is a weak homotopy equivalence, it induces a simple homotopy equivalence $|\k (X\smallsetminus \{x\})|\to |\k (X)|$. In particular $X\smallsetminus \{x\} \se X$.
\end{teo} 
\begin{proof}
Since $|\k(X\smallsetminus \{x\})|$ is a strong deformation retract of $|\k (X)|$ and the open star of $x$, $$\overset{\circ}{st}(x)=|\k(X)|\smallsetminus |\k (X\smallsetminus \{x\})|$$ is contractible, then by \cite[(20.1)]{Coh}, the Whitehead Torsion $\tau (|\k (X)|, |\k (X\smallsetminus \{x\})|)=0$.
\end{proof}

This result essentially shows that one-point reductions are not sufficient to describe all weak homotopy types of finite spaces. Of course they are sufficient to reach all finite models of spaces with trivial Whitehead group. On the other hand, note that the fact that $X\smallsetminus \{x\}$ and $X$ have the same weak homotopy type does not imply that the inclusion $X\smallsetminus \{x\} \hookrightarrow X$ is a weak homotopy equivalence.

\begin{defi}
If $x\in X$ is a $\gamma$-point, we say that there is an \textit{elementary $\gamma$-collapse} from $X$ to $X\smallsetminus \{x\}$. A finite $T_0$-space $X$ \textit{$\gamma$-collapses} to $Y$ if there is a sequence of elementary $\gamma$-collapses that starts in $X$ and ends in $Y$. We denote this by $X\gammac Y$. If $X$ $\gamma$-collapses to a point, we say that it is $\gamma$-collapsible.
\end{defi}

In contrast to collapses, a $\gamma$-collapse does not induce in general a collapse between the associated simplicial complexes. For example, if $K$ is any triangulation of the Dunce hat, $\mathbb{C}(\x (K))\gammac$ $\x (K)$, but $aK' \ / \! \! \! \! \! \! \searrow K'$ since $K'$ is not collapsible (see Lemma \ref{mig}).

We finish this section analyzing the relationship between $\gamma$-collapsibility and joins.

\begin{prop} \label{join}
Let $X$ and $Y$ be finite $T_0$-spaces. Then
\begin{enumerate}
\item[(\emph{i})] $X \oplus Y$ is homotopically trivial if $X$ or $Y$ is homotopically trivial.
\item[(\emph{ii})] $X \oplus Y$ is $\gamma$-collapsible if $X$ or $Y$ is $\gamma$-collapsible.
\end{enumerate}
\end{prop} 
\begin{proof}
If $X$ or $Y$ is homotopically trivial, $|\k (X)|$ or $|\k (Y)|$ is contractible and then so is $|\k (X)|*|\k (Y)|=|\k (X\oplus Y)|$. Therefore $X\oplus Y$ is homotopically trivial.

The proof of $(ii)$ follows as in Proposition \ref{joincontr}. If $x_i\in X_i$ is a $\gamma$-point, $\hat{C}_{x_i}^{X_i\oplus Y}=\hat{C}_{x_i}^{X_i} \oplus Y$ is homotopically trivial by item $(i)$ and then $x_i$ is a $\gamma$-point of $X_i\oplus Y$.
\end{proof}

There is an analogous result for acyclic spaces that follows from the K\"unneth formula for joins \cite{Mil2}. 

Note that the converse of these results are false. To see this, consider two finite simply connected simplicial complexes $K, \ L$ such that $H_2(|K|)=\mathbb{Z}_2$, $H_2(|L|)=\mathbb{Z}_3$ and $H_n(|K|)=H_n(|L|)=0$ for every $n\ge 3$. Then $\x (K)$ and $\x (L)$ are not acyclic, but $\x (K) \oplus \x (L)$, which is weak homotopy equivalent to $|K|*|L|$, is acyclic by the K\"unneth formula and, since it is simply connected (see \cite{Mil2} or Remark \ref{libre}), it is homotopically trivial.

A counterexample for the converse of item $(ii)$ is the following.

\begin{ej} \label{ejemplo1}
Let $K$ be a triangulation of the Dunce hat. Then, $\x (K)$ is a homotopically trivial finite space of height 3. The non-Hausdorff suspension $\S (\x (K))=\x (K) \cup \{-1,1\}$ is $\gamma$-collapsible since $1$ is a $\gamma$-point and $\S (\x (K)) \smallsetminus \{1\}$ has maximum. However $\x (K)$ is not collapsible and then $\S (\x (K))$ is not collapsible by Proposition \ref{joincol}. Moreover $\x (K)$ and $S^0$ are not $\gamma$-collapsible either because their heights are less than or equal to 3.
\end{ej}

%The importance of $\gamma$-collapses is that they provide a method of reduction for finite spaces, this is to say that if $X\gammac Y$, $X$ and $Y$ are weak homotopy equivalent.

\section{h-regular complexes}

Recall that a CW-complex $K$ is regular if for each (open) cell $e^n$, the characteristic map $D^n\to \overline{e^n}$ is a homeomorphism, or equivalently, the attaching map $S^{n-1}\to K$ is a homeomorphism onto its image $\dot{e}^n$, the boundary of $e^n$. In this case, it can be proved that the closure $\overline{e^n}$ of each cell is a subcomplex, which is equivalent to say that $\dot{e}^n$ is a subcomplex.

A cell $e$ of a regular complex $K$ is a face of a cell $e'$ if $e\subseteq \overline{e'}$. This will be denoted by $e\le e'$.
The barycentric subdivision $K'$ is the simplicial complex whose vertices are the cells of $K$ and whose simplices are the sets $\{e_1,e_2, \ldots ,e_n\}$ such that $e_i$ is a face of $e_{i+1}$.

We can define, as in the case of simplicial complexes, the face poset $\x (K)$ of a regular complex $K$, which is the set of cells ordered by $\le$. Note that $\k (\x (K))=K'$, which is homeomorphic to $K$ and therefore $\x (K)$ is a finite model of $K$, i.e. it has the same weak homotopy type as $K$.

\begin{ej}
The following figure (Fig. 3) shows a regular structure for the real projective plane $\mathbb{R}P^2$. The edges are identified in the way indicated by the arrows. It has three 0-cells, six 1-cells and four 3-cells. Therefore its face poset has 13 points (Fig. 4). 

\begin{minipage}{6.5cm}
\qquad
\qquad
\includegraphics[scale=0.4]{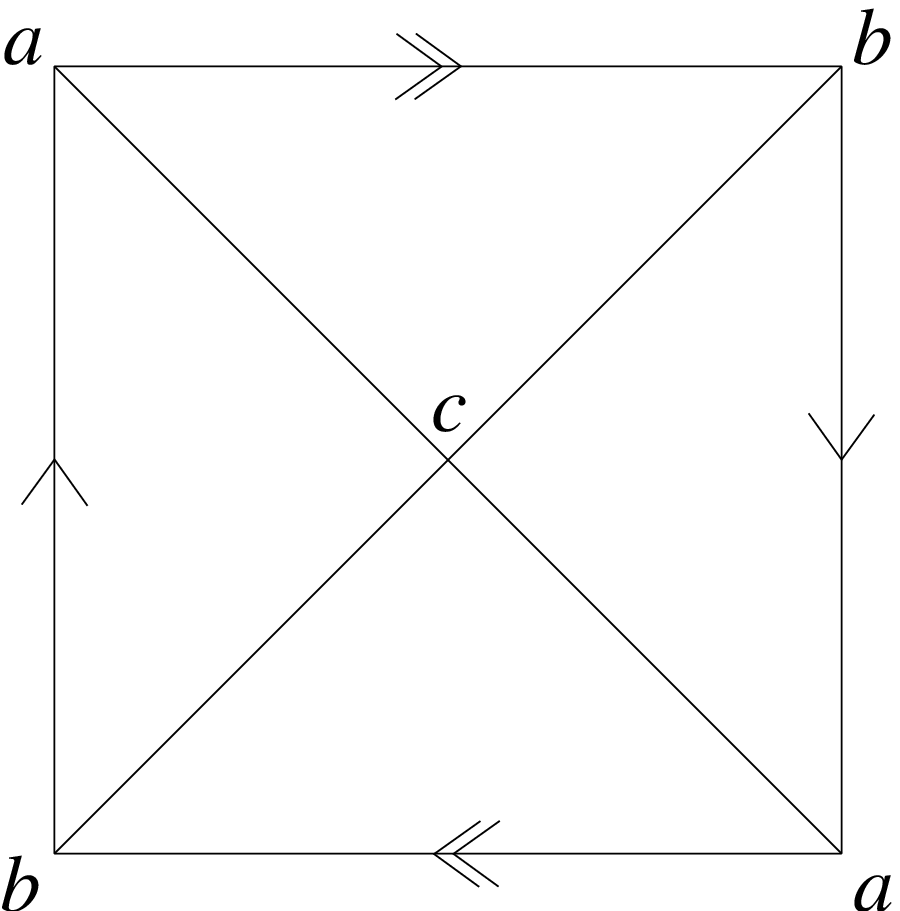}
\center{Fig. 3}
\end{minipage}
\begin{minipage}{7cm}
\qquad
\qquad
\smallskip
\medskip
\medskip
\begin{displaymath}
\xymatrix@C=14pt{ \bullet \ar@{-}[d] \ar@{-}[drr] \ar@{-}[drrrrr] & & \bullet \ar@{-}[ld] \ar@{-}[d] \ar@{-}[rrrrd] & & \bullet \ar@{-}[llld] \ar@{-}[d] \ar@{-}[dr] & & \bullet \ar@{-}[dllllll] \ar@{-}[dll] \ar@{-}[d] \\
								\bullet \ar@{-}[d] \ar@{-}[drrr] & \bullet \ar@{-}[ld] \ar@{-}[rrd] & \bullet \ar@{-}[lld] \ar@{-}[drrrr] & & \bullet \ar@{-}[dllll] \ar@{-}[drr] & \bullet \ar@{-}[dll] \ar@{-}[dr] & \bullet \ar@{-}[dlll] \ar@{-}[d] \\
		_a \bullet \ & & & _b \bullet \ & & & \ \bullet _c} 
\end{displaymath}
\center{\raisebox{-3ex}{Fig. 4}}
\end{minipage}
\end{ej}

In this section we introduce the concept of  \textit{h-regular complex}, generalizing the notion of regular complex. Given an h-regular complex $K$, one can define $\x(K)$ as before. In general, $K$ and $\k(\x(K))$ are not homeomorphic. However we prove that $\x (K)$ is a finite model of $K$. 
We also study the relationship between collapses of h-regular complexes and of finite spaces.

\begin{defi}
A CW-complex $K$ is \textit{h-regular} if the attaching map of each cell is a homotopy equivalence with its image and the closed cells $\overline{e^n}$ are subcomplexes of $K$.
\end{defi}

In particular, regular complexes are h-regular.

\begin{prop}
Let $K=L\cup e^n$ be a CW-complex such that $\dot{e}^n$ is a subcomplex of $L$. Then $\overline{e^n}$ is contractible if and only if the attaching map $\varphi :S^{n-1} \to \dot{e}^n$ of the cell $e^n$ is a homotopy equivalence.
\end{prop}
\begin{proof}
Suppose $\varphi : S^{n-1} \to \dot{e}^n$ is a homotopy equivalence. Since $S^{n-1} \hookrightarrow D^n$ has the homotopy extension property, the characteristic map $\psi :D^n \to \overline{e^n}$ is also a homotopy equivalence.

Suppose now that $\overline{e^n}$ is contractible. 
The map $\overline{\psi}:D^n / S^{n-1}\to \overline{e^n} / \dot{e}^n$ is a homeomorphism and therefore it induces isomorphisms in homology and, since $\overline{e^n}$ is contractible, by the long exact sequence of homology it follows that $\varphi _* :H_k(S^{n-1})\to H_k(\dot{e}^n)$ is an isomorphism for every $k$.

If $n\ge 3$, $\pi_1 (\dot{e}^n)=\pi_1 (\overline{e^n})=0$ and by a theorem of Whitehead, $\varphi$ is a homotopy equivalence.
If $n=2$, $\dot{e}^n$ is just a graph and since $\varphi _* :H_1(S^1)\to H_1(\dot{e}^n)$ is an isomorphism, the attaching map $\varphi$ is a homotopy equivalence.
Finally, if $n=1$, since the cell is contractible, $\varphi$ is one-to-one and therefore a homeomorphism.
\end{proof}

\begin{coro}
A CW-complex is h-regular if and only if the closed cells are contractible subcomplexes. 
\end{coro}

\begin{ej} \label{dh}
The following are four different h-regular structures for the Dunce hat which are not regular structures. In each example the edges are identified in the way indicated by the arrows.

\begin{minipage}{8cm}
\qquad
\includegraphics[scale=0.4]{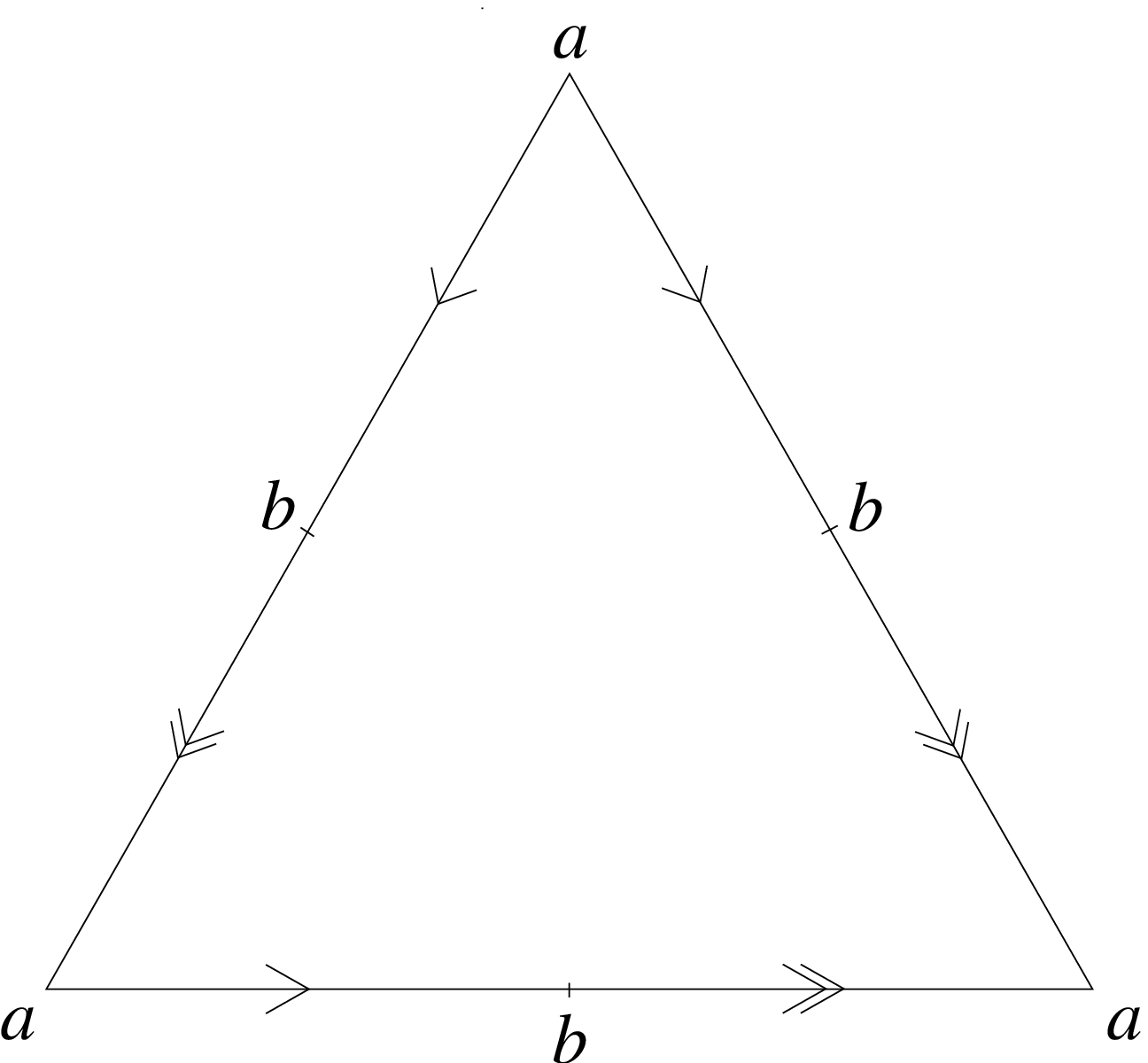}
\end{minipage}
\begin{minipage}{6cm}
\includegraphics[scale=0.4]{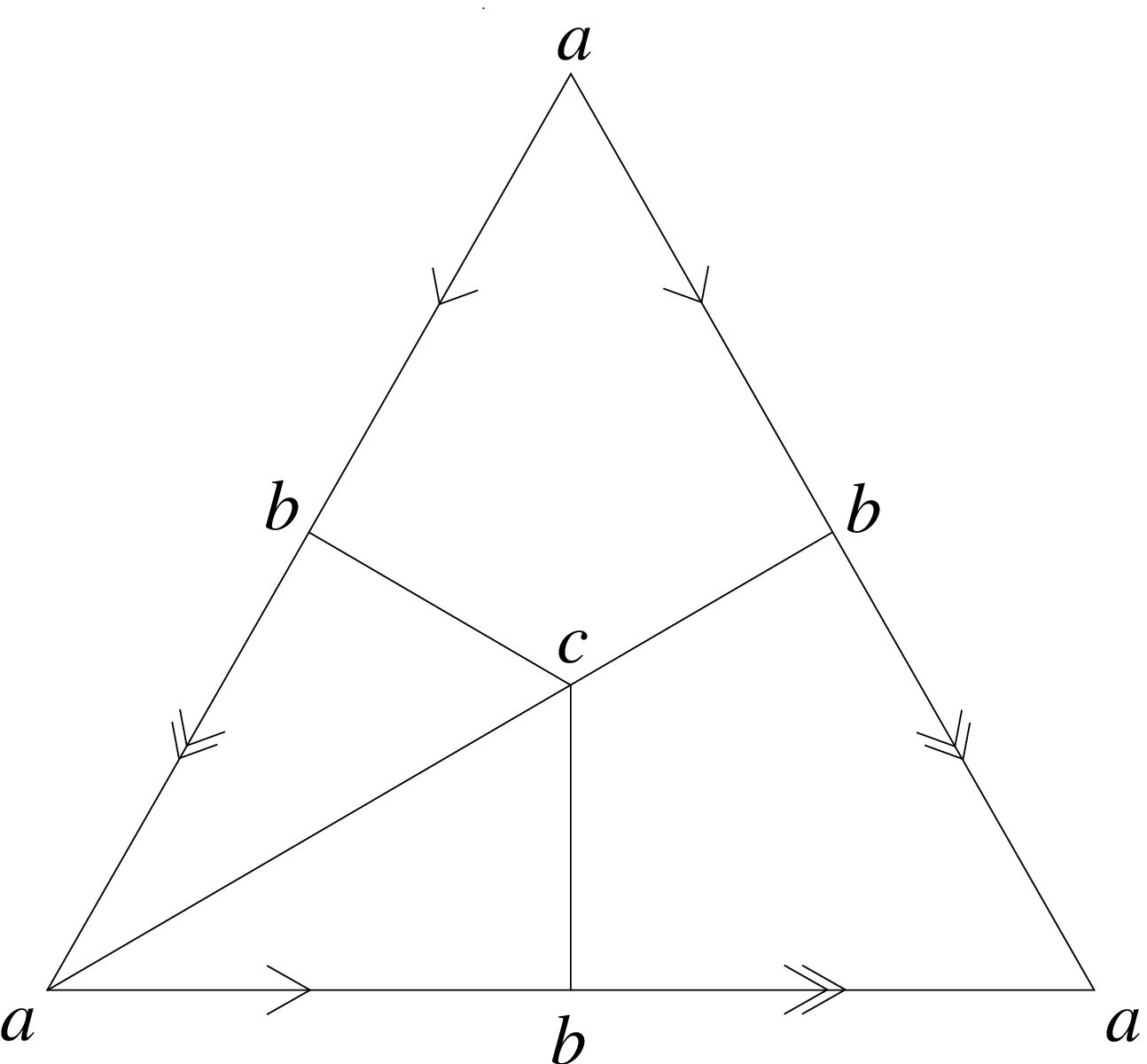}
\end{minipage}

\begin{minipage}{8cm}
\qquad
\includegraphics[scale=0.4]{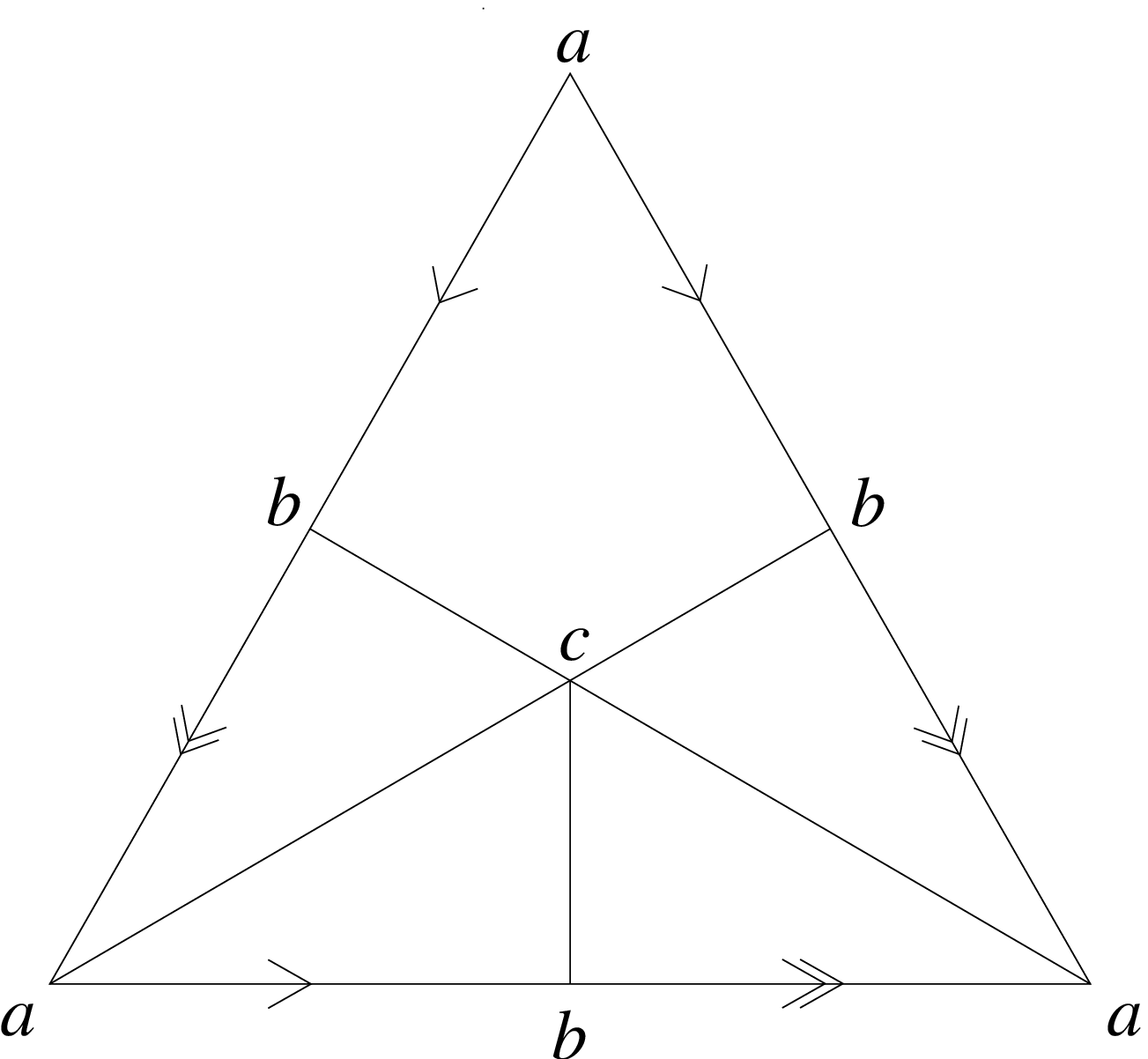}
\end{minipage}
\begin{minipage}{6cm}
\includegraphics[scale=0.4]{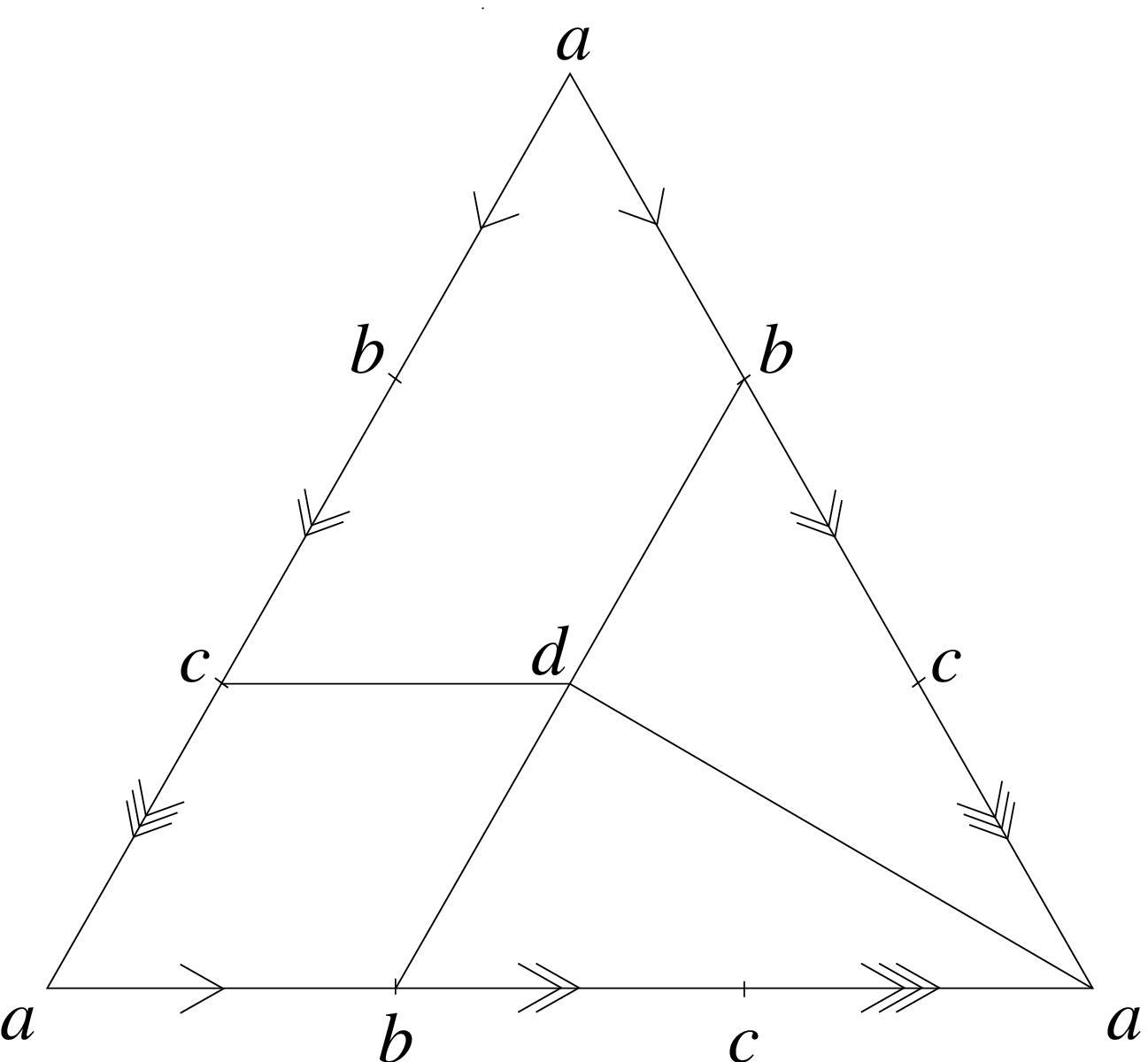}
\end{minipage}
\end{ej}

For an h-regular complex $K$, we also define the \textit{associated finite space} (or \textit{face poset}) $\x (K)$  as the poset of cells of $K$ ordered by the face relation $\le$, like in the regular case.

The proof of the following lemma is standard.

\begin{lema}
Let $K\cup e$ be a CW-complex, let $\psi :D^n\to \overline{e}$ be the characteristic map of the cell $e$ and let $A$ be a subspace of $\dot{e}$. We denote $C_e(A)=\{ \psi (x) \ | \ x\in D^n\smallsetminus\{0\}, \ \psi (\frac{x}{\parallel x \parallel}) \in A \} \subseteq \overline{e}$. Then
\begin{enumerate}
\item If $A\subseteq \dot{e}$ is open, $C_e(A)\subseteq \overline{e}$ is open.
\item $A\subseteq C_e(A)$ is a strong deformation retract.
\end{enumerate}
\end{lema}

\begin{teo}
If $K$ is a finite h-regular complex, $\x(K)$ is a finite model of $K$.
\end{teo} 
\begin{proof}
We define recursivelly a weak homotopy equivalence $f_K:K\to \x (K)$.

Assume $f_{K^{n-1}}:K^{n-1}\to \x (K^{n-1})\subseteq \x(K)$ is already defined and let $x=\psi (a)$ be a point in an $n$-cell $e^n$ with characteristic map $\psi :D^n\to \overline{e^n}$. If $a=0\in D^n$, define $f_K(x)=e^n$. Otherwise, define $f_K(x)=f_{K^{n-1}}(\psi(\frac{a}{\parallel a \parallel }))$. 

In particular note that if $e^0\in K$ is a $0$-cell, $f_K(e^0)=e^0\in \x (K)$. Notice also that if $L$ is a subcomplex of $K$, $f_L=f_K|_L$.

We will show by induction on the number of cells of $K$, that for every cell $e\in K$, $f_K^{-1}(U_e)$ is open and contractible. This will prove that $f_K$ is continuous and, by McCord's Theorem 
\cite[Theorem 6]{Mcc}, a weak homotopy equivalence.

Let $e$ be a cell of $K$. Suppose first that there exists a cell of $K$ which is not contained in $\overline{e}$. Take a maximal cell $e'$ (with respect to the face relation $\le$) with this property. Then $L=K\smallsetminus e'$ is a subcomplex and by induction, $f_L^{-1}(U_e)$ is open in $L$. It follows that $f_L^{-1}(U_e)\cap \dot{e}' \subseteq \dot{e}'$ is open and by the previous lemma, $C_{e'}(f_L^{-1}(U_e)\cap \dot{e}') \subseteq \overline{e'}$ is open. Therefore, $$f_K^{-1}(U_e)=f_L^{-1}(U_e)\cup C_{e'}(f_L^{-1}(U_e)\cap \dot{e}')$$ is open in $K$.

Moreover, since $f_L^{-1}(U_e)\cap \dot{e}' \subseteq C_{e'}(f_L^{-1}(U_e)\cap \dot{e}')$ is a strong deformation retract, so is $f_L^{-1}(U_e)\subseteq f_K^{-1}(U_e)$. By induction, $f_K^{-1}(U_e)$ is contractible.

In the case that every cell of $K$ is contained in $\overline{e}$, $f_K^{-1}(U_e)=\overline{e}=K$, which is open and contractible.
\end{proof}   

As an application we deduce that the finite spaces associated to the h-regular structures of the Dunce hat considered in Example \ref{dh} are all homotopically trivial. 
The first one is a contractible space of 5 points, the second one is a collapsible and non-contractible space of 13 points and the last two are non-collapsible spaces of 15 points since they do not have weak points.
Here we exhibit the Hasse diagram of the space associated to the third h-regular structure of the Dunce hat.

\begin{displaymath}
\xymatrix@C=16pt@R=36pt{ \bullet \ar@{-}[d] \ar@{-}[dr] \ar@{-}[drrr]& \bullet \ar@{-}[d] \ar@{-}[dr] \ar@{-}[drrrr]& & \bullet \ar@{-}[dlll] \ar@{-}[dl] \ar@{-}[dr]& & \bullet \ar@{-}[dl] \ar@{-}[d] \ar@{-}[dr] & \bullet \ar@{-}[dlll] \ar@{-}[dl] \ar@{-}[d]  \\
		\bullet \ar@{-}[dr] \ar@{-}[drrr] & \bullet \ar@{-}[drr] \ar@{-}[drrrr] & \bullet \ar@{-}[dl] \ar@{-}[drrr] & \bullet \ar@{-}[dll] \ar@{-}[drr] & \bullet \ar@{-}[dlll] \ar@{-}[dr] & \bullet \ar@{-}[dllll] \ar@{-}[dll] & \bullet \ar@{-}[dlll] \ar@{-}[dl] \\
		& _b \bullet \ & & \ \bullet _a & & \ \bullet _c & } 
\end{displaymath}
\begin{center}
Fig. 5: A homotopically trivial non-collapsible space of 15 points.
\end{center}

\begin{ej}
Let $K$ be the space which is obtained from a square by identifying all its edges as indicated.
\begin{center} \includegraphics[scale=0.4]{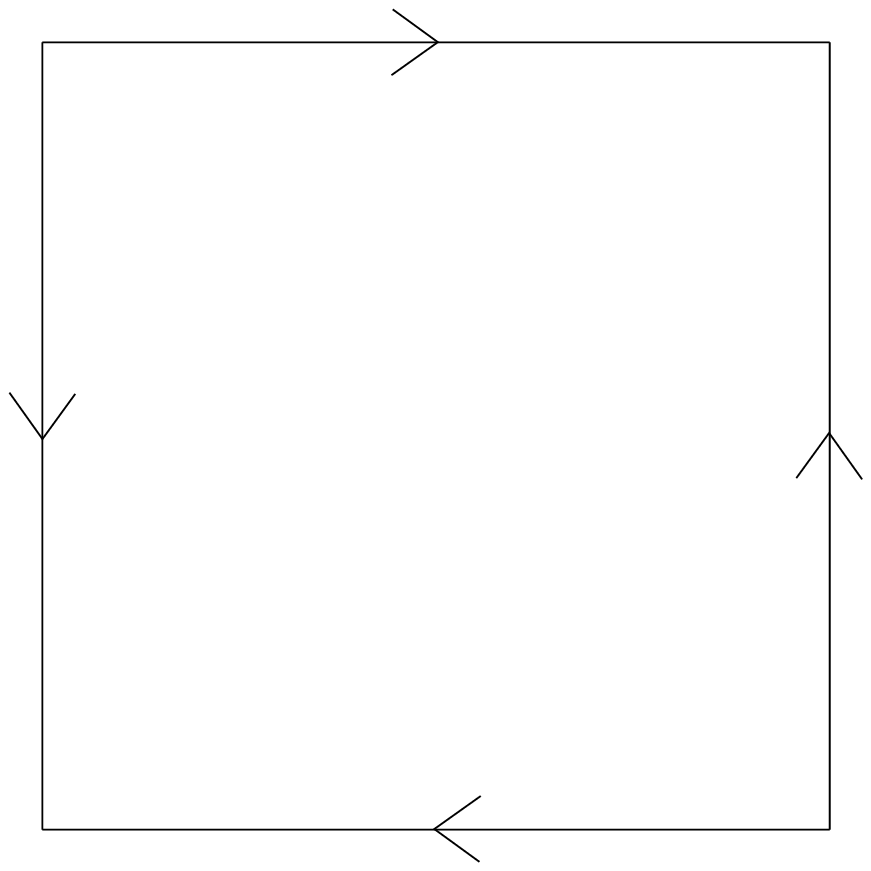} \end{center}
We verify that K is homotopy equivalent to $S^2$ using techniques of finite spaces.
Consider the following h-regular structure of $K$
\begin{center} \includegraphics[scale=0.4]{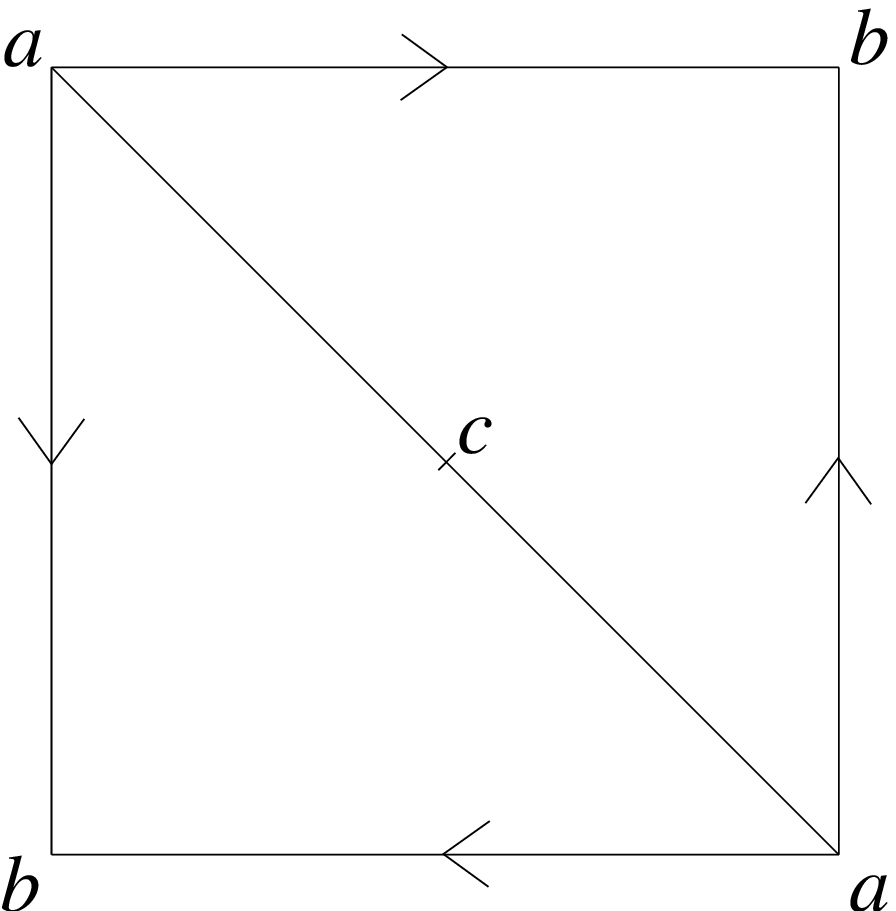} \end{center}
which consists of three 0-cells, three 1-cells and two 2-cells.
The Hasse diagram of the associated finite space $\x (K)$ is
\begin{displaymath}
\xymatrix@C=6pt{ \bullet \ar@{-}[d] \ar@{-}[drr] \ar@{-}[drrr] & & \bullet \ar@{-}[dll] \ar@{-}[d] \ar@{-}[dr] & & \\ 
		\bullet \ar@{-}[d] \ar@{-}[drr] & & \bullet \ar@{-}[dll] \ar@{-}[d] & \ \ \raisebox{0.9ex}{$\bullet ^{ab}$} \ar@{-}[dl] \ar@{-}[dr] & \\ 
		_c \bullet \ & & \ \bullet _a & & \ \bullet _b } 
\end{displaymath}
The 0-cell $b$ is an up beat point of $\x (K)$ and the 1-cell $ab$ is a down beat point of $\x (K) \smallsetminus \{b\}$. Therefore $K$ is weak homotopy equivalent to $\x (K) \smallsetminus \{b,ab\}$ which is a (minimal) finite model of $S^2$ (see \cite{Bar}). In fact $\x (K) \smallsetminus \{b,ab\}=S^0 \oplus S^0 \oplus S^0$ is weak homotopy equivalent to $S^0*S^0*S^0=S^2$.
\end{ej}

In \cite{Bar2} we proved that a collapse $K \searrow L$ of finite simplicial complexes induces a collapse $\x(K) \searrow \x(L)$ between the associated finite spaces. This is not true when $K$ and $L$ are regular complexes. Consider $L=\k (W)$ the associated simplicial complex to the Wallet $W$ (see Fig. 6 below), and $K$ the CW-complex obtained from $L$ by attaching a regular 2-cell $e^2$ with boundary $\k (\{a,b,c,d\})$ and a regular 3-cell $e^3$ with boundary $L\cup e^2$.

\begin{displaymath}
\xymatrix@C=6pt{ \bullet \ar@{-}[d] \ar@{-}[drr] & & ^a \bullet \ \ar@{-}[lld] \ar@{-}[rrd] & & \ \ \bullet ^b \ar@{-}[lld] \ar@{-}[rrd] & & \bullet \ar@{-}[lld] \ar@{-}[d]  \\
		\bullet \ar@{-}[dr] \ar@{-}[drrr] & & \bullet \ar@{-}[dl] \ar@{-}[dr] & & \bullet \ar@{-}[dl] \ar@{-}[dr] & & \bullet \ar@{-}[dlll] \ar@{-}[dl] \\
		& _c \bullet \ & & \bullet & & \ \bullet _d } 
\end{displaymath}
\begin{center}
Fig. 6: $W$
\end{center}
Note that the complex $K$ is regular and collapses to $L$, but $\x(K)=\x(L) \cup \{e^2,e^3\}$ does not collapse to $\x(L)$ because $\hat{U}_{e^3}^{\x(K)\smallsetminus \{e^2\}}=\x (L)=W'$ is not contractible.
However, one can prove that a collapse $K\searrow L$ between h-regular CW-complexes induces a $\gamma$-collapse $\x (K) \gammac \x (L)$.

\begin{teo} \label{coll}
Let $L$ be a subcomplex of an h-regular complex $K$. If $K\searrow L$, then $\x(K) \gammac \x(L)$.
\end{teo}
\begin{proof}
Assume $K=L\cup e^n \cup e^{n+1}$. Then $e^n$ is an up beat point of $\x (K)$.
Since $K\searrow L$, $\overline{e^{n+1}}\searrow L\cap \overline{e^{n+1}}= \dot{e}^{n+1} \smallsetminus e^n$. In particular $\dot{e}^{n+1} \smallsetminus e^n$ is contractible and then $$\hat{C}_{e^{n+1}}^{\x (K)\smallsetminus \{e^n\}}=\x (\dot{e}^{n+1} \smallsetminus e^n)$$ is homotopically trivial. Therefore  $$\x (K)\ce \x(K)\smallsetminus \{e^n\}\gammac \x(L).$$
\end{proof}

\end{document}